\documentclass[onecolumn,9pt]{elsarticle}

\usepackage{url}
\usepackage{natbib}            
\usepackage{graphicx}          
\usepackage{slashbox}                               

\usepackage{amssymb}
\usepackage{amsmath}
\usepackage{graphicx}
\usepackage{amsfonts}
\usepackage{multirow}
\usepackage{mathtools}

\usepackage{psfrag}
\usepackage{mathrsfs}
\DeclareMathOperator{\diag}{diag}

\everymath{\displaystyle}

\newtheorem{theorem}{Theorem}[section]
\newtheorem{corollary}[theorem]{Corollary}

\newtheorem{lemma}[theorem]{Lemma}
\newtheorem{define}[theorem]{Definition}

\newenvironment{proof}{{\it Proof :~}}{\hfill$\diamondsuit$\\}
\newtheorem{remark}{Remark}
\newtheorem{example}{Example}

\DeclareMathOperator*{\col}{col}
\DeclareMathOperator*{\He}{Sym}

\DeclareMathOperator{\eps}{\varepsilon}

\DeclareMathOperator{\co}{co}

\begin{document}

\begin{frontmatter}

\title{A looped-functional approach for robust stability analysis of linear impulsive systems}

\author[First]{Corentin Briat}\ead{briatc@bsse.ethz.ch,corentin@briat.info}\ead[url]{http://www.briat.info}
\author[Second]{Alexandre Seuret}\ead{Alexandre.Seuret@gipsa-lab.grenoble-inp.fr}\ead[url]{http://www.gipsa-lab.fr/\textasciitilde alexandre.seuret}

\address[First]{Swiss Federal Institute of Technology--Z\"{u}rich (ETH-Z), Department of Biosystems Science and Engineering (D-BSSE), Switzerland.}
\address[Second]{CNRS, GIPSA-Lab, Control Systems Department, France.}

\begin{keyword}
Impulsive systems, looped-functionals, Uncertain systems, LMIs
\end{keyword}

\begin{abstract}
A new functional-based approach is developed for the stability analysis of linear impulsive systems. The new method, which introduces looped-functionals, considers non-monotonic Lyapunov functions and leads to LMIs conditions devoid of exponential terms. This allows one to easily formulate dwell-times results, for both certain and uncertain systems. It is also shown that this approach may be applied to a wider class of impulsive systems than existing methods. Some examples, notably on sampled-data systems, illustrate the efficiency of the approach.
\end{abstract}

\end{frontmatter}

\section{Introduction}

Impulsive systems \citep{Bainov:89, Yang:01b, Cai:05, Cai:08, Hespanha:08, Michel:08, Goebel:09} are an important class of hybrid systems admitting discontinuities in the state trajectories, that are governed by discrete-time maps. They occur in several fields like epidemiology \citep{Stone:00,Briat:09h}, sampled-data and networked control systems \citep{Naghshtabrizi:08}, forestry \cite{Verriest:09d}, power electronics \cite{Loxton:09}, harvesting problems \cite{Yu:06,Lin:11}, etc. Among the wide class of impulsive dynamical systems, we may distinguish between systems whose impulse-times depend on the system state and those for which the impulse-times are external to system and only time-dependent. The latter class may be represented in the following form
\begin{equation}\label{eq:mainsyst}
  \begin{array}{rcl}
    \dot{x}(t)&=&Ax(t),\ t\in\mathbb{R}_+\backslash\mathbb{I}\\
    x^+(t)&=&Jx(t),\ t\in\mathbb{I}\\
    x(t_0)&=&x_0
  \end{array}
\end{equation}
where $x,x_0\in\mathbb{R}^n$ are the state of the system and the initial condition, respectively. The system matrices are possibly uncertain, i.e.  $(A,J)\in\mathcal{A}\times\mathcal{J}$ where
\begin{equation}\label{eq:polytopes}
\begin{array}{lclcl}
  \mathcal{A}:=\co\left\{A_1,\ldots,A_{N_A}\right\},&&A_i\in\mathbb{R}^{n\times n},\ i=1,\ldots,N_A\\
  \mathcal{J}:=\co\left\{J_1,\ldots,J_{N_J}\right\},&&J_i\in\mathbb{R}^{n\times n},\ i=1,\ldots,N_J
\end{array}
\end{equation}
where $N_A,N_J\in\mathbb{N}\backslash\{0\}$ and $\co\{\cdot\}$ is the convex hull. The state trajectory is assumed to be left-continuous and to have right-limits at all time. The notation $x^+(t)$ denotes the right limit of $x(s)$ as $s$ tends to $t$ from the right, i.e. $x^+(t)=\lim_{s\downarrow t}x(s)$. The set of impulsive times $\mathbb{I}:=\{t_k\}_{k\in\mathbb{N}}$ is a countable set of impulse instants $t_{k+1}>t_k$, $k\in\mathbb{N}$ and we define the inter-impulse distance as ${T_k:=t_{k+1}-t_k}$. This quantity is also referred to as \emph{dwell-time} in the literature. We also assume that the sequence $\{t_k\}_{k\in\mathbb{N}}$ has no accumulation point in order to exclude any Zeno behavior. 

Depending on the structure of the matrices $A$ and $J$, the system
may exhibit very different behaviors. In particular, notions of
minimal and maximal dwell-time can be defined for impulsive systems
\citep{Hespanha:08}, similarly as for switched systems
\citep{Geromel:06b, Liberzon:03}. In the case of impulsive systems,
these notions refer to system properties such that too large or too
short dwell-times destabilize the system. In the case of
periodic impulses with period $T>0$, the problem essentially reduces
to the study of the Schurness of the matrix $Je^{AT}$, which turns
out to be a very simple problem. However, this formulation suffers
from several critical drawbacks:
\begin{enumerate}
  \item The eigenvalue analysis is not extendable to the case of aperiodic impulses since the spectral radius is, in general, not submultiplicative;
  \item To overcome this, Lyapunov approaches can be applied but lead to robust Linear Matrix Inequalities (LMIs) with scalar uncertainties at the exponential, known to be complex numerically, yet solvable \citep{Oishi:10};
  \item The extension to robust stability analysis is also difficult, again due to the exponential terms. There is no efficient solution, at this time, to handle block matrix uncertainties at the exponential.
\end{enumerate}
The approach discussed in this paper aims at overcoming the above
important drawbacks and brings an efficient solution for the
characterization of robust-dwell times. The main tool used in this
paper is an extension of the one developed in \citep{Seuret:12} in
the context of sampled-data systems, itself triggered by a somewhat
different approach discussed in the anterior work
\citep{Naghshtabrizi:08} where impulsive systems are considered for
the representation of aperiodic sampled-data systems\citep{Sun:91, Sivashankar:94, Dullerud:99}. The core of the approach relies on the implicit but equivalent correspondence between
discrete- and continuous-time domains obtained in \citep{Seuret:12}.
It is shown there that discrete-time stability is equivalent to a
certain kind of continuous-time stability, provided that the latter
is proved using particular functionals referred to as \emph{looped-functionals}. One important feature of
pointwise stability criteria is to allow for non-monotonic
continuous-time Lyapunov functions (evaluated along the flow a system) since only a pointwise decrease is imposed instead of a continuous one. A discrete-time criterion
provides then a weaker condition for stability than classical
continuous-time ones. In the case of impulsive systems, by studying
the decrease of the Lyapunov function evaluated at impulse-times
only, non-monotonicity of the continuous-time Lyapunov function
between impulse-times may be tolerated.

The interesting point of the method is its wide adaptability to any type of systems having discrete-events, or more generally time-marks, which are time-instants for which the system admits certain regularity properties. For sampled-data and impulsive systems, these time-marks coincide with the control law update and impulse-times, respectively. Hence, by looking at a discrete-time Lyapunov function whose values at the marker-times are decreasing allows one to prove stability of the overall hybrid system. This concept is immediately extendable to other type of systems for which such marks may be defined.

The looped-functional-based approach introduced above leads to LMI-conditions that are \emph{affine in the dwell-time}, \emph{devoid of exponential terms} and \emph{able to consider the jumps precisely} through the non-monotonicity of the Lyapunov function. Indeed, expansive jumps on a certain state can be tolerated as long as it decreases sufficiently between jumps. Conversely, increasing between jumps is also allowed as long as the jumps are contracting. The obtained LMI conditions easily allow one to both characterize stability under periodic and aperiodic impulses. In both cases, the dwell-time $T_k=t_{k+1}-t_k$ plays an important role in the stability of the impulsive system. Several stability concepts are considered in this paper: stability with minimal dwell-time $T_k\in[T_{min},+\infty)$ where $T_{min}$ is the so-called \emph{minimal dwell-time}, stability with maximal dwell-time $T_k\in(0,T_{max}]$ where $T_{max}$ is the \emph{maximal-dwell time}, stability with ranged dwell-time $T_k\in[T_{min},T_{max}]$ and stability with arbitrary pulsing $T_k>0$. 
The domains of application of these different dwell-time concepts are summarized in Table \ref{tab:app}. It is important to stress that the case (3,3) cannot be handled using existing approaches, e.g. \citep{Cai:08,Hespanha:08}, since none of the system matrices are stable, while the proposed approach does not require such stability conditions on $A$ and $J$. This emphasizes that the proposed approach better captures the internal structure of the system. Thanks to the affine dependence of the obtained conditions on the system matrices, dwell-time notions are finally extended to the case of uncertain systems, a problem for which there is currently a lack of efficient solution techniques.

It seems important to stress that since the paper has been submitted, several improvements have been made. A more advanced explicit looped-functional has been proposed in \cite{Briat:12d}. An approach based on an implicit looped-functional has also been discussed in \cite{Briat:13a}.\\

\begin{table}
\centering
  \begin{tabular}{c||c|c|c|}
   \backslashbox{J}{A} & $\Re[\lambda(A)]<0$ & $\Re[\lambda(A)]>0$ & otherwise\\
    \hline
    \hline
    \multirow{2}{*}{$\rho(J)<1$} & arbitrary & \multirow{2}{*}{maximal DT} & ranged DT\\
    & minimal DT &           &     maximal DT  \\
    \hline
    $\rho(J)>1$ & minimal DT & --- & ---\\
    \hline
    otherwise & minimal DT & --- & ranged DT\\
    \hline
  \end{tabular}
  \caption{Application table for the developed results (DT stands for 'Dwell-Time')}\label{tab:app}
\end{table}

\noindent\textbf{Outline:} Section \ref{sec:prel} introduces a generalization of the result of \citep{Seuret:12}, at the core of the approach. Sections \ref{sec:nom} and \ref{sec:rob} are devoted to nominal and robust stability analysis using a class of looped-functionals fulfilling the conditions stated by the main theorem of Section \ref{sec:prel}. Illustrative examples are included in the related sections.\\

\noindent\textbf{Notations:} For symmetric matrices $A,B$, $A-B\prec(\preceq)0$ means that $A-B$ is negative (semi)definite. The sets of symmetric and positive definite matrices of dimension $n$ are denoted by $\mathbb{S}^n$ and $\mathbb{S}_{++}^n$ respectively. The spectral radius and the spectrum are denoted by $\rho(\cdot)$ and $\lambda(\cdot)$ respectively. Given a square matrix $A$, we define $\He[A]=A+A^T$.

\section{Preliminary definitions and results}\label{sec:prel}

\subsection{Lifting}

The functional based approach relies on the characterization of system (\ref{eq:mainsyst}) trajectories in a lifted domain, similar to the one used in sampled-data systems theory \citep{Yamamoto:90}, with the difference that the involved functions do not have identical support. Indeed, we view here the entire state-trajectory as a sequence of functions
$$\left\{x(t_k+\tau),\ \tau\in(0,T_k]\right\}_{k\in\mathbb{N}}.$$
The elements of the sequence have unique continuous extension to $[0,T_k]$ defined as
\begin{equation}
\begin{array}{rcl}
  \chi_k(\tau)&:=&x(t_k+\tau),\\
  \chi_k(0)&=&\lim_{s\downarrow t_k}x(s).
\end{array}
\end{equation}
We also have the following identities
\begin{equation}
  \begin{array}{rcl}
    \chi_k(\tau)&=&e^{A\tau}\chi_k(0),\\
    \chi_{k+1}(0)&=&J\chi_k(T_k).
  \end{array}
\end{equation}
Hence, in the following, the state-space of the impulsive system will be defined as the union set of continuous-functions
$$\mathbb{K}_{[T_{min},T_{max}]}:= \smashoperator{\bigcup_{T\in[T_{min},T_{max}]}}\left\{\vphantom{\sum} C([0,T],\mathbb{R}^n)\right\}$$
with support in a certain range. This varying support is necessary to consider the {aperiodicity} of the system. Note that when $T_{min}=T_{max}=\bar{T}$, the usual lifting space of periodic sampled-data systems is recovered \citep{Yamamoto:90} and degenerates to $\mathbb{K}_{\bar{T}}:=\mathbb{K}_{[\bar{T},\bar{T}]}=C([0,\bar{T}],\mathbb{R}^n)$.



\subsection{Stability analysis - Looped-functional-based results}

The definition below introduces the asymptotic stability of impulsive systems:
\begin{define}
  Given an increasing sequence of impulse instants $\{t_k\}_{k\in\mathbb{N}}$ having no accumulation point, the system (\ref{eq:mainsyst}) is globally asymptotically stable if for all $x_0\in\mathbb{R}^n$ and all $\delta>0$, there exists $\delta_0>0$ such that
  \begin{itemize}
\item $||x_0||\le\delta_0\Rightarrow||x(t)||\le\delta$ for all $t\ge t_0$ (stability),
\item $||x(t)||\to0$ as $t\to\infty$ (attractivity).
  \end{itemize}
  Alternatively, the asymptotic stability can also be defined in the state-space $\mathbb{K}_{[T_{min},T_{max}]}$ where the stability and attractivity definitions become
    \begin{itemize}
    \item $||x_0||\le\delta_0\Rightarrow\sup_{s\in[0,T_k]}||\chi_k(s)||\le\delta$ for all $k\in\mathbb{N}$,
    \item $\sup_{s\in[0,T_k]}||\chi_k(s)||\to0$ as $k\to\infty$.
  \end{itemize}
\end{define}

It is important to stress that the above stability definition strongly depends on the considered impulse sequence and may be very difficult to apply. The following weaker definition addresses the case where an entire family of impulse instants is considered:
\begin{define}
  The system is said to be globally asymptotically stable under ranged dwell-time if for any increasing sequence of impulse instants $\{t_k\}_{k\in\mathbb{N}}$ having no accumulation point and satisfying $t_{k+1}-t_k\in[T_{min},T_{max}]$, $T_{min}>\eps>0$, $k\in\mathbb{N}$, the system (\ref{eq:mainsyst}) is globally asymptotically stable.
\end{define}

The above definitions are clearly not the most general stability notions for hybrid systems, but are sufficient for the considered problem. The system being linear, properties such as stability are automatically global. Secondly, since accumulation points in the sequence of impulse times are excluded, the standard asymptotic stability notion is meaningful. It is hence not necessary to define hybrid time domains and the positive real axis (or the set of natural numbers) can be used to denote time. More general stability notions for hybrid dynamical systems, such as pre-asymptotic stability, can be found in \citep{Goebel:09}.

The following technical definition is necessary before stating the next very important result.
\begin{define}
  A functional $f:[0,T_{max}]\times\mathbb{K}_{[T_{min},T_{max}]}\times[T_{min},T_{max}]\to\mathbb{R}$, $\eps\le T_{min}\le T_{max}<+\infty$, $\eps>0$, is said to be a \textbf{looped-functional} if the following conditions are satisfied
\begin{enumerate}
  \item the equality
  \begin{equation}\label{Th02}
    f(0,z,T)=f(T,z,T)
  \end{equation}
  holds for all functions $z\in C([0,T],\mathbb{R}^n)\subset\mathbb{K}_{[T_{min},T_{max}]}$ and all $T\in[T_{min},T_{max}]$,
  and
  \item it is differentiable with respect to the first variable with the standard definition of the derivative.
\end{enumerate}
The set of all such functionals is denoted by $\mathfrak{LF}([T_{min},T_{max}])$.
\end{define}

The idea for proving stability of (\ref{eq:mainsyst}) is to look at the behavior of a candidate discrete-time Lyapunov function evaluated at the impulse instants, that is, we look for a positive definite quadratic form $V(x)$ such that the sequence $\{V(\chi_k(T_k))\}_{k\in\mathbb{N}}$ is monotonically decreasing\footnote{We may also look at the sequence $\{V(\chi_k(0))\}_{k\in\mathbb{N}}$ instead. The choice is purely arbitrary.}. This is formalized below through a functional existence result:
\begin{theorem}\label{Th0}
Let $\eps<T_{min}\leq T_{max}$ be three finite positive scalars and $V:
\mathbb R^n\rightarrow \mathbb R_+$ be a quadratic form verifying
\begin{equation}
\forall x\in \mathbb R^n,\quad
\mu_1||x||_2^2\leq  V(x)\leq \mu_2||x||_2^2\label{Th01}
\end{equation}
for some scalars $0<\mu_1\le\mu_2$. Assume that one of the following equivalent statements hold:
\begin{description}
\item[(i)] The sequence $\{V(\chi_k(T_k))\}_{k\in\mathbb{N}}$ is decreasing; that is $V(x)$ is a discrete-time Lyapunov function for system $x(t_{k+1})=e^{AT_k}Jx(t_k)$, $T_k\in[T_{min},T_{max}]$.

\item[(ii)] There exists a looped-functional $\mathcal{V}\in\mathfrak{LF}([T_{min},T_{max}])$ such that the  functional $\mathcal{W}_k$ defined for any $c_k\in\mathbb{R}$ as
\begin{equation}\label{eq:functdsq}
  \begin{array}{rcl}
    \mathcal{W}_k(\tau,\chi_k,\chi_{k-1})&:=&\dfrac{\tau}{T_k}\Lambda_k+V(\chi_{k}(\tau))+\mathcal{V}(\tau,\chi_k,T_k)+c_k,\\
  \end{array}
\end{equation}
where $$\Lambda_k=V(\chi_k(0))-V(\chi_{k-1}(T_{k-1})), $$ has a
derivative along the trajectories of system
$\dot{\chi}_k(\tau)=A\chi_k(\tau)$, $\tau\in[0,T_k]$
\begin{equation}\label{Th03}
\dfrac{d}{d\tau}\mathcal{W}_k(\tau,\chi_k,\chi_{k-1}):=\dfrac{1}{T_k}\Lambda_k+\dfrac{d}{d\tau}V(\chi_k(\tau))+\dfrac{d}{d\tau}\mathcal{V}(\tau,\chi_k,T_k)
\end{equation}
which is negative definite for all $\tau\in(0,T_k)$, $T_k\in[T_{min},T_{max}]$, $k\in\mathbb{N}$.
\end{description}

Then, the solutions of system (\ref{eq:mainsyst}) with known matrices $A$ and $J$ are asymptotically stable for any sequence of impulse time instants $\{t_k\}_{k\in\mathbb{N}}$ satisfying $t_{k+1}-t_k\in[T_{min},T_{max}]$, $k\in\mathbb{N}$.
\end{theorem}
\begin{proof}
\textbf{Proof of (ii)$\Rightarrow$(i):} Let $k\in\mathbb N$, $\tau\in(0,T_k]$ and $T_k\in[T_{min},T_{max}]$. Assume that $(ii)$ is satisfied. Integrating (\ref{Th03}) over $[0,T_k]$ yields
$$\begin{array}{lcl}
\int_{0}^{T_k}\dfrac{d}{d\tau}\mathcal{W}_k(\tau,\chi_k,\chi_{k-1})d\tau&=&\mathcal{W}_k(T_k,\chi_k,\chi_{k-1})-\mathcal{W}_k(0,\chi_k,\chi_{k-1})\\
&=&V(\chi_k(0))-V(\chi_{k-1}(T_{k-1}))\\
&&+V(\chi_k(T_k))-V(\chi_k(0))\\
&&+\mathcal V(T_k,\chi_k,T_k)-\mathcal V(0,\chi_k,T_k).
\end{array}$$
The terms on the last row vanish according to (\ref{Th02}) and we get
$$\begin{array}{lcl}
\mathcal{W}_k(T_k,\chi_k,\chi_{k-1})-\mathcal{W}_k(0,\chi_k,\chi_{k-1})
&=&V(\chi_k(T_k))-V(\chi_{k-1}(T_{k-1})).
\end{array}$$
Then, the sequence $\{V(\chi_k(T_k))\}_{k\in\mathbb{N}}$ is decreasing over $k$ since (\ref{Th03}) is negative over $[0,T_k]$.\\

\textbf{Proof of (i)$\Rightarrow$(ii):}
Assume that $(i)$ is satisfied. Similarly as in \citep{Seuret:12}, introduce the functional $\mathcal{V}(\tau,\chi_k,T_k)=-V(\chi_{k}(\tau))+\dfrac{\tau}{T_k}\left[V(\chi_k(T_k))-V(\chi_k(0))\right]$. It is immediate to see that $\mathcal{V}\in\mathfrak{LF}([T_{min},T_{max}])$ since we have
\begin{equation}
  \begin{array}{lcl}
    \mathcal{V}(T_k,\chi_k,T_k)&=&-V(\chi_{k}(T_k))+(V(\chi_k(T_k)))-V(\chi_k(0))\\
    &=&-V(\chi_{k}(0))\\
    &=&\mathcal{V}(0,\chi_k,T_k)
  \end{array}
\end{equation}
for all $T_k\in[T_{min},T_{max}]$. Substitution of the proposed functional $\mathcal{V}$ into (\ref{Th03}) yields
$$\dfrac{d}{d\tau}\mathcal{W}_k(\tau,\chi_k,\chi_{k-1})=\dfrac{1}{T_k}\left[V(\chi_k(T_k))-V(\chi_{k-1}(T_{k-1}))\right]$$
and is negative by assumption. Equivalence is hence proved.\\

\textbf{Proof of asymptotic stability:}
It remains to prove that convergence of the discrete-time Lyapunov function implies boundedness and asymptotic convergence of the continuous-time Lyapunov function, or equivalently of the convergence of $||x(t)||_2$ to 0. From the discrete-time Lyapunov condition we have that $||\chi_k(T_k)||_2\to0$ as $k\to\infty$. Note also that
\begin{equation}
  \begin{array}{lcl}
 V(\chi_{k+1}(\tau)) &=& \chi_{k+1}(\tau)^TP\chi_{k+1}(\tau)\\
                 &=& \chi_{k}(T_{k})^TJ^Te^{A^T\tau}Pe^{A\tau}J\chi_k(T_k)\\
                 &\le&\eta||\chi_k(T_k)||_2^2\\
                 &\le&\eta\mu_1^{-1}V(\chi_k(T_k))
  \end{array}
\end{equation}
where $\eta:=\sup_{s\in[T_{min},T_{max}]}\bar{\lambda}\left(J^Te^{A^Ts}Pe^{As}J\right)$ and $\bar{\lambda}$ is the largest eigenvalue. Note that $T_{max}$ may be allowed to be unbounded provided that $A$ is Hurwitz since, in that case, $\eta$ would be bounded. Finally, from the boundedness and convergence of $V(\chi_k(T_k))$, we have the boundedness and convergence of $V(x(t))$ and $||x(t)||_2$. This completes the proof.

%
%
\end{proof}

The interest for considering discrete-time Lyapunov functions lies in
the potential use of non-monotonic continuous-time Lyapunov functions (along the trajectories of the system) that are
impossible to consider via the usual Lyapunov Theorem. Indeed,
despite being non-monotonic, Lyapunov functions may asymptotically
tend to 0, and thus contain information on the asymptotic stability
of the system. This feature is extremely important in the current
framework in order to cope with expansive jumps and unstable
continuous-time dynamics. Using such a discrete-time approach, only
the decreasing of the function evaluated at instants
$\{t_k\}_{k\in\mathbb{N}}$ is important. In Fig. \ref{fig:loc2}, we
may see that the two envelopes, generated by the pre-impulses and
post-impulses values of the continuous-time Lyapunov function,
characterize asymptotic stability. The functional $\mathcal{W}$ of
Theorem \ref{Th0} coincides with the lower envelope and a specific
sequence $\{c_k\}_{k\in\mathbb{N}}$, making the envelope absolutely
continuous. Note that the lower and upper envelope are equivalent in
terms of stability measure since they are related through the
equality $V(\chi_k(0))=V(J\chi_{k-1}(T_{k-1}))$.

\begin{figure}[h]
\centering
\includegraphics[width=0.65\textwidth]{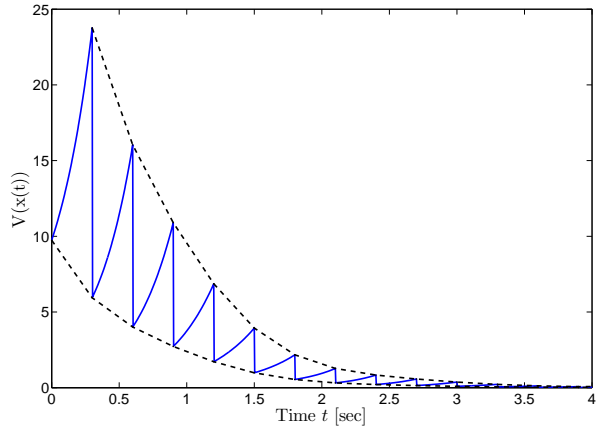}
  \caption{Continuous-time Lyapunov function $V(x)=x^TPx$ (plain) for system (\ref{eq:ex1}) and the discrete-time envelopes (dashed); $\mathcal{W}$ coincides with the monotonically decreasing lower envelope.}\label{fig:loc2}
\end{figure}

According to the choice of the sequence $\{c_k\}_{k\in\mathbb{N}}$,
the functional $\mathcal{W}$ can be discontinuous. This, however, is
not a problem since the functional decrease must hold over the
intervals $(t_k,t_{k+1})$. It is important to keep in mind that only
the integral of the derivative over $(t_k,t_{k+1})$ is important
since it coincides with a pointwise decrease of the Lyapunov function
at impulse instants.

It will be shown that such an approach is able to characterize stability for a larger class of systems than existing methods \citep{Cai:08,Hespanha:08}, more precisely, systems for which neither $A$ nor $J$ are stable. The approach also readily extends to the case of uncertain systems for which efficient results are needed.

\section{Nominal stability analysis of linear impulsive systems}\label{sec:nom}

This section provides some nominal stability results for both periodic and aperiodic impulses. Therefore, we tacitly assume in this section that the sets $\mathcal{A}$ and $\mathcal{J}$ are reduced to the singletons $\{A\}$ and $\{J\}$, respectively. In the following, we will make extensive uses of the matrix expressions:
\begin{equation}
\begin{array}{rcl}
  \mathscr{C}(P,A)&:=&A^TP+PA,\\
  \mathscr{D}(P,J)&:=&J^TPJ-P,\\
  \mathscr{I}(P,A,J,T)&:=&J^Te^{A^TT}Pe^{AT}J-P,
\end{array}
\end{equation}
where $\mathscr{C}$, $\mathscr{D}$ and $\mathscr{I}$ stand for 'continuous', 'discrete' and 'impulsive', respectively.

\subsection{Stability analysis of linear impulsive systems - Periodic impulses case}

The simple case of periodic linear impulsive systems is considered first to set up ideas and key results on minimal dwell-time, maximal-dwell-time, arbitrary impulses and ranged dwell-time. In the periodic impulse case, the LTI continuous-time impulsive system (\ref{eq:mainsyst}) can be converted into the LTI discrete-time system
\begin{equation}\label{eq:dt}
  x(t_{k+1})=e^{AT}Jx(t_k),
\end{equation}
the asymptotic stability of which is equivalent to the asymptotic stability of (\ref{eq:mainsyst}). This is formalized in the following lemma:
\begin{lemma}[Periodic impulses]\label{lem:periodic}
   Assume $T>0$, then the following statements are equivalent:
   \begin{enumerate}
     \item[a)] The LTI continuous-time impulsive system (\ref{eq:mainsyst}) with $T$-periodic impulses is asymptotically stable.
     \item[b)] The LTI discrete-time system (\ref{eq:dt}) is asymptotically stable.
     \item[c)] There exists a matrix $P=P^T\succ0$ such that
  \begin{equation}\label{eq:lolperiod}
    \mathscr{I}(P,A,J,T)\prec0.
  \end{equation}
   \end{enumerate}
\end{lemma}

The main drawback of the above results lies in the presence of exponential terms in LMI (\ref{eq:lolperiod}) preventing any extension to robust dwell-time characterization, essentially due to the difficulty of considering uncertainties at the exponential. Note also that the LMI (\ref{eq:lolperiod}) is difficult to check when $T$ belongs to a (possibly infinite) interval, as this is the case for certain dwell-time notions.

Below, we show that Theorem \ref{Th0} may be used to obtain sufficient conditions for the asymptotic stability of system (\ref{eq:mainsyst}) in the case of periodic impulses. The same result turns out to be useful for deriving alternative dwell-time results and generalizing them to uncertain systems by resolving the exponential-related tractability problem mentioned above.
\begin{theorem}\label{th:imp:nominal}
  The impulsive system (\ref{eq:mainsyst}) with $T$-periodic impulses is asymptotically stable if there exist matrices $P,Z\in\mathbb{S}_{++}^n$, $Q,U\in\mathbb{S}^n$, $R\in\mathbb{R}^{n\times n}$ and $N\in\mathbb{R}^{n\times 2n}$ such that the LMIs
\begin{equation}\label{eq:th1}
\begin{array}{lcccl}
  \Psi(T)&:=&F_0+T(F_2+F_3)&\prec&0\\
  \Phi(T)&:=&\begin{bmatrix}
    F_0-TF_3 & N^T\\
    \star & -\frac{Z}{T}
  \end{bmatrix}&\prec&0
  \end{array}
\end{equation}
hold with $M_x=\begin{bmatrix}
  I & 0
\end{bmatrix}$, $M_\zeta=\begin{bmatrix}
  I & -J
\end{bmatrix}$, $M_-=\begin{bmatrix}
  0 & I
\end{bmatrix}$, $F_3=M_-^TJ^TUJM_-$ and
\begin{equation}
  \begin{array}{lcl}
    F_0&=&TM_x^T(A^TP+PA)M_x-M_\zeta^TQM_\zeta\\
    &&+M_-^T(J^TPJ-P)M_-+\He[N^TM_\zeta-M_\zeta^TRM_x]\\
    F_2&=&\He[M_x^TA^TQM_\zeta+M_x^TA^TRM_x+M_\zeta^TRAM_x]\\
    &&+M_x^TA^TZAM_x.
  \end{array}
\end{equation}
Moreover, the quadratic form $V(x)=x^TPx$ is a discrete-time Lyapunov function for system (\ref{eq:mainsyst}), that is the LMI $\mathscr{I}(P,A,J,T)\prec0$ holds, implying then the satisfaction of the conditions of Lemma \ref{lem:periodic}.
\end{theorem}
\begin{proof}
The proof is inspired from \citep{Seuret:12}. Choosing $V(x)=x^TPx$ and
\begin{equation}\label{eq:calV}
\begin{array}{lcl}
    \mathcal{V}(\tau,\chi_k,T)&=&\dfrac{(T-\tau)}{T}\zeta_k(\tau)^T\left[Q\zeta_k(\tau)+2R\chi_k(\tau)\right]\\
    &&+\dfrac{(T-\tau)}{T}\int_{0}^{\tau}\dot{\chi}_k(s)^TZ\dot{\chi}_k(s)ds\\
    &&+\dfrac{\tau(T-\tau)}{T}\chi_{k}(0)^TU\chi_{k}(0)
\end{array}
\end{equation}
where $\zeta_k(\tau)=\chi_k(\tau)-\chi_k(0)$, $\dot{\chi}_k(\tau)=A\chi_k(\tau)$, $P\in\mathbb{S}_{++}^n,Z\in\mathbb{S}_{++}^n,Q\in\mathbb{S}^n,U\in\mathbb{S}^n$ and $R\in\mathbb{R}^{n\times n}$. The condition (\ref{Th02}) is verified since
\begin{equation}
  \begin{array}{lclclcl}
    \mathcal{V}(0,z,T)&=&\mathcal{V}(T,z,T)&=&0
  \end{array}
\end{equation}
for all $z\in C([0,T],\mathbb{R}^n)$. Thus, according to Theorem \ref{Th0}, the functional $\mathcal{W}_k$ defined in (\ref{eq:functdsq}) must be considered. Its derivative (\ref{Th03})  taken along the trajectories of the system is bounded from above by the quadratic form
\begin{equation}
\begin{array}{lcl}
  \dot{\mathcal{W}}_k&\le&\dfrac{1}{T}\xi_k(\tau)^T\left[F_0+\tau F_1+(T-\tau)F_2 +(T-2\tau)F_3\right]\xi_k(\tau)
\end{array}
\end{equation}
where $F_1=N^TZ^{-1}N$ and $\xi_k(t)=\col(\chi_k(\tau),\chi_{k-1}(T))$. The above bound on $\dot{\mathcal{W}}_k$ has been obtained using the affine Jensen's bound \citep{Seuret:09b,Briat:11b} on the integral term
\begin{equation}\label{eq:affjen}
\begin{array}{lcl}
    -\int_{0}^{\tau}\dot{\chi}_k(s)^TZ\dot{\chi}_k(s)(s)ds&\le&\xi_k(\tau)^T\left(2N^TM_\zeta+\tau N^TZ^{-1}N\right)\xi_k(\tau)
\end{array}
\end{equation}
defined for some $N\in\mathbb{R}^{n\times 2n}$. This bound is known to be more precise and relevant than the rational Jensen's inequality when applied on integrals with uncertain/varying integration bounds \cite{Briat:11b}.
Hence, the system (\ref{eq:mainsyst}) with periodic impulses is
asymptotically stable if $\dot{\mathcal{W}}_k$ is negative definite
over $\tau\in(0,T)$, or equivalently, if the LMI
$$F_0+\tau F_1+(T-\tau)F_2+(T-2\tau)F_3\prec0$$ holds for all $\tau\in[0,T]$. Since this LMI is affine in $\tau$ (hence convex), to check its negative definiteness over the entire interval $(0,T]$, it is necessary and sufficient to check it at the vertices of the set, that is only over the \emph{finite set} $\tau\in\{0,T\}$. A Schur complement on the quadratic term $TN^T Z^{-1} N$ then yields the result.
\end{proof}

\begin{remark}[Note on the affine Jensen's bound]
  By eliminating the slack-matrix $N$ on the RHS of (\ref{eq:affjen}) using the Finsler's lemma \citep{SkeltonIG:97a} the usual Jensen's bound is recovered, showing then their equivalence. However, as pointed out in \cite{Briat:11b}, they are not equivalent from a computational point of view. Indeed, the rational Jensen's bound is nonconvex in the measure of the integration support (here $\tau$) and must be bounded by its maximal value, i.e. $T$, resulting then in a conservatism increase and a loss of equivalence. On the other hand, the affine Jensen's bound is affine in the measure of the integration support (hence convex) and does not need to be overbounded in order to make the conditions tractable.
\end{remark}

It is important to stress that while Theorem \ref{Th0} provides a necessary and sufficient condition, the above result provides a sufficient one only. By indeed choosing the specific functional (\ref{eq:calV}), necessity is destroyed. Despite of that, it will be shown in the examples that such a functional may yield interesting results. To anticipate a little bit, the functional-based formulation will also be shown in Section \ref{sec:rob} to be suitable for robust stability analysis. All these characteristics emphasize the relevance of the proposed approach inheriting part of the efficiency
of the usual discrete-time Lyapunov approach.

We digress here a little bit to address the case of stability under small periods, i.e. the case $T\to0$, but excluding any Zeno behavior.
\begin{lemma}\label{lem:T0}
  When $T\to0$, then the LMI conditions (\ref{eq:th1}) tend to the condition $\mathscr{D}(P,J)\prec0$, equivalent to the Schurness of $J$. This is consistent with the fact that $\rho(Je^{AT})\to\rho(J)$ as $T\to0$.
\end{lemma}
\begin{proof}
 When $T\to0$, then we have $\Psi(T)\to\left.F_0\right|_{T=0}$ and $\Phi(T)\to\left.F_0\right|_{T=0}$ where $\left.F_0\right|_{T=0}$ is given by
\begin{equation}
\begin{array}{l}
  -M_\zeta^TQM_\zeta-M_\zeta^TRM_x-M_x^TR^TM_\zeta\\
  \qquad+N^TM_\zeta+M_\zeta^TN+M_-^T(J^TPJ-P)M_-
\end{array}
\end{equation}
and should be negative definite. By virtue of the Finsler's Lemma \citep{SkeltonIG:97a}, the matrix $N$ can be eliminated and we obtain the equivalent LMI
\begin{equation}
  M_\bot^TM_-^T(J^TPJ-P)M_-M_\bot\prec0
\end{equation}
where $M_\bot=\begin{bmatrix}
  J^T&I
\end{bmatrix}^T$. Evaluation of the expression yields the result.
\end{proof}

It is important to stress that $T\to0$ does not necessarily means that an accumulation point exists. The limit in the small period allows one to study the stability for very small periods that are still different from 0. Let us consider for instance the unbounded sequence of impulse times given by $t_k=\log(k+1)$, $k\in\mathbb{N}$. The dwell-times are then given by $T_k:=t_{k+1}-t_k=\log(1+1/(k+1))$ and tends to 0 as $k$ goes to infinity. In this case, stability of the system must be ensured for arbitrarily small positive dwell-times.

Thus, from Lemma \ref{lem:T0} and by virtue of the continuity of eigenvalues with respect to matrix parameters, the condition $\rho(J)<1$ implies $\rho(e^{AT}J)<1$ in a sufficiently small positive neighborhood of $T=0$. This result will turn out to be very useful in Section \ref{sec:ranged} on the ranged dwell-time.

\subsection{Minimal dwell-time characterization}

Several results characterizing stability with minimal dwell-time and arbitrary impulse sequences are provided below. The two results considering the minimal-dwell-time are obtained using similar arguments as in \cite{Geromel:06b}, and using Theorem \ref{Th0}, respectively. The case of arbitrary impulse sequences is a particular case of the first result.

\begin{lemma}[Minimal Dwell-Time]\label{lem:minDT}
  Assume that for some given $T>0$, there exists a matrix $P=P^T\succ0$ such that the LMIs
  \begin{equation}\label{eq:lolilol1}
    \mathscr{C}(P,A)\prec0
  \end{equation}
  and
  \begin{equation}\label{eq:lolilol2}
    \mathscr{I}(P,A,J,T)\prec0
  \end{equation}
  hold. Then, for any impulse sequence $\{t_k\}_{k\in\mathbb{N}}$ satisfying $t_{k+1}-t_k\ge T$, the system (\ref{eq:mainsyst}) is asymptotically stable.
\end{lemma}
\begin{proof}
The goal is to show that the conditions (\ref{eq:lolilol1}) and (\ref{eq:lolilol2}) implies that $ \mathscr{I}(P,A,J,s)\prec0$ for all $s\ge T$. To this aim, we use the fact that the eigenvalues of $J^Te^{A^Ts}Pe^{As}J$ are nonincreasing as $s>0$ increases.
To show this, let us consider the Lyapunov function evaluated at $t_k+s$, $s\ge0$ given by
  \begin{equation}
         V(x(t_k+s))=x(t_k)^TJ^Te^{A^Ts}Pe^{As}Jx(t_k)
  \end{equation}
  whose derivative with respect to $s$ is given by
  \begin{equation}
     \dfrac{d}{ds}V(x(t_k+s))=x(t_k)^TJ^Te^{A^Ts}\mathscr{C}(P,A)e^{As}Jx(t_k).
  \end{equation}
By virtue of the condition (\ref{eq:lolilol1}), the derivative is negative semidefinite for all $s\ge0$ and hence the eigenvalues of $J^Te^{A^Ts}Pe^{As}J$ are nonincreasing as $s$ increases. We then have the inequality
$$J^Te^{A^Ts}Pe^{As}J\preceq J^Te^{A^TT}Pe^{AT}J$$
for all $s\ge T$. Using now the condition (\ref{eq:lolilol2}) this implies that we have
\begin{equation}
  \mathscr{I}(P,A,J,s)\preceq \mathscr{I}(P,A,J,T)\prec0
\end{equation}
for all $s\ge T$. The proof is complete.
\end{proof}

Similarly as in the periodic case, we derive below an affine counterpart of the above result, which will be useful for the minimal dwell-time analysis of uncertain systems.
\begin{theorem}[Minimal Dwell-Time]\label{th:minDT}
  Assume that for some given $T>0$, there exist matrices $P,Z\in\mathbb{S}^n_{++}$, $U,Q\in\mathbb{S}^n$, $N\in\mathbb{R}^{n\times 2n}$ and $R\in\mathbb{R}^{n\times n}$ such that the LMIs $\Psi(T)\prec0$, $\Phi(T)\prec0$ and $\mathscr{C}(P,A)\prec0$ hold. 

Then, for any impulse sequence $\{t_k\}_{k\in\mathbb{N}}$ satisfying $t_{k+1}-t_k\ge T$, the system (\ref{eq:mainsyst}) is asymptotically stable and Lemma \ref{lem:minDT} holds with the same matrix $P$.
\end{theorem}
\begin{proof}
  The proof follows from Lemma \ref{lem:minDT} and Theorem \ref{th:imp:nominal}.
\end{proof}

Note that when deriving Lemma \ref{lem:minDT}, there is no restriction in letting $T\to0$, leading then to a possible characterization of the stability for arbitrarily small dwell-times. Lemma \ref{lem:T0} says that if there exists $P\in\mathbb{S}_{++}^n$ such that $\mathscr{D}(J,P)\prec0$ holds, then asymptotic stability is guaranteed for sufficiently small dwell-times. The following corollary is a consequence of this fact and the maximal dwell-time result of Lemma \ref{lem:minDT}:
\begin{corollary}[Arbitrary impulses sequence]\label{cor:arb}
  The system (\ref{eq:mainsyst}) is asymptotically stable for arbitrary impulse sequence $\{t_k\}_{k\in\mathbb{N}}$ verifying $t_{k+1}-t_k>\eps$, for any $0<\eps<T$ and for all $k\in\mathbb{N}$, if there exists a matrix $P\in\mathbb{S}_{++}^n$ such that $\mathcal{C}(P,A)\prec0$ and $\mathscr{D}(P,J)\prec0$ hold.
\end{corollary}
\begin{proof}
It is sufficient to let $T\to0$ in the condition (\ref{eq:lolilol2}) of Lemma \ref{lem:minDT}.
\end{proof}

Note that the same condition can be found in \citep[Theorem 4.1.1]{Yang:01b} or in \citep[Theorem 2]{Hespanha:08} albeit expressed in different ways. 

\subsection{Maximal dwell-time characterization and arbitrary impulses}

Results on stability with maximal dwell-time \citep{Hespanha:08} are discussed in this section.

\begin{lemma}[Maximal Dwell-Time]\label{lem:maxDT}
  Assume that for some given $T>0$, there exists a matrix $P\in\mathbb{S}_{++}^n$ such that the LMIs $\mathcal{C}(P,A)\succ0$ and ${\mathscr{I}(P,A,J,T)\prec0}$ hold.

  Then, for any impulse sequence $\{t_k\}_{k\in\mathbb{N}}$ satisfying $\eps<t_{k+1}-t_k\le T$, for any $0<\eps<T$, the system (\ref{eq:mainsyst}) is asymptotically stable.
\end{lemma}
\begin{proof}
  The proof is similar to the one of Lemma \ref{lem:minDT}.
\end{proof}

\begin{theorem}[Maximal Dwell-Time]\label{th:maxDT}
  Assume that for some given $T>0$, there exist matrices $P,Z\in\mathbb{S}^n_{++}\succ0$, $U,Q\in\mathbb{S}^n$, $N\in\mathbb{R}^{n\times 2n}$ and $R\in\mathbb{R}^{n\times n}$ such that the LMIs $\Psi(T)\prec0$, $\Phi(T)\prec0$ and $\mathscr{C}(P,A)\succ0$ hold. 

  Then, for any impulse sequence $\{t_k\}_{k\in\mathbb{N}}$ satisfying $\eps<t_{k+1}-t_k\le T$, for any $0<\eps<T$, the system (\ref{eq:mainsyst}) is asymptotically stable and Lemma \ref{lem:maxDT} holds with the same matrix $P$.
\end{theorem}
\begin{proof}
  The proof follows from Lemma \ref{lem:maxDT} and Theorem \ref{th:imp:nominal}.
\end{proof}

\subsection{Stability over arbitrary intervals - Ranged dwell-time}\label{sec:ranged}

The above results can only be applied when $A$ is Hurwitz or anti-Hurwitz. To overcome this limitation, the following results considering dwell-times belonging to general bounded intervals of the form $[T_{min},T_{max}]$ are provided. This leads to the notion of \emph{ranged dwell-time} which combines the concepts of minimal and maximal dwell-time together.

It is immediate to derive the following preliminary result on ranged dwell-time:
\begin{lemma}[Ranged dwell-time]\label{th:ranged}
   Assume there exists a matrix $P\in\mathbb{S}_{++}^n$ such that \begin{equation}\label{eq:tyouliol}
    \mathscr{I}(P,A,J,\theta)\prec0
\end{equation}
for all $\theta\in[T_{min},T_{max}]$, $\eps<T_{min}<T_{max}<+\infty$, for some $\eps>0$.

Then, for any impulse sequence $\{t_k\}_{k\in\mathbb{N}}$ satisfying $t_{k+1}-t_k\in[T_{min},T_{max}]$, the system (\ref{eq:mainsyst}) is asymptotically stable.
\end{lemma}
It is clear that the semi-infinite dimensional LMI (\ref{eq:tyouliol}) is not easy to check due to the presence of the uncertain parameter $\theta\in[T_{min},T_{max}]$. The use of Theorem \ref{Th0} allows one to overcome this difficulty.
\begin{theorem}\label{th:aperiodic_range}
  The impulsive system (\ref{eq:mainsyst}) with $T_k\in[T_{min},T_{max}]$, $\eps<T_{min}\le T_{max}<\infty$, $\eps>0$, is asymptotically stable if there exist matrices $P,Z\in\mathbb{S}_{++}^n$, $Q,U\in\mathbb{S}^n$, $R\in\mathbb{R}^{n\times n}$ and $N\in\mathbb{R}^{n\times 2n}$ such that $\Psi(T)\prec0$ and $\Phi(T)\prec0$
hold for all $T\in\{T_{min},T_{max}\}$.

Moreover, in such a case, the inequality $\mathscr{I}(P,A,J,\theta)\prec0$ holds for all ${\theta\in[T_{min},T_{max}]}$ and Lemma \ref{th:ranged} is verified with the same matrix $P$.
\end{theorem}
\begin{proof}
  Let us consider the LMIs of Theorem \ref{th:imp:nominal} allowing the check the stability of the system for a constant dwell-time $T>0$. In order to prove that the impulsive system is asymptotically stable for a varying dwell-time belonging to $[T_{min},T_{max}]$, the LMI must be simply checked over the interval $\theta\in[T_{min},T_{max}]$. Since the LMIs $\Psi(T)$ and $\Phi(T)$ are convex in $T$, it is then enough to check it at the vertices of the interval, and hence for all values in the finite set $\{T_{min},T_{max}\}$. This concludes the proof.
\end{proof}

It seems important to point out that, unlike the results on minimal and maximal dwell-time requiring stability or anti-stability of the matrix $A$, the above result is applicable to any linear impulsive system. It can notably be applied to systems for which neither $A$ nor $J$ is stable, which is very interesting since existing methods cannot handle such a situation. This is to put in contrast, for instance, with \citep[Theorem 1]{Hespanha:08}, transferred to a linear setting, for which it is necessary that at least one of the matrices be stable. Another feature is the reduction to a finite-dimensional LMI problem while the condition of Lemma \ref{th:ranged} is semi-infinite dimensional.

Notably, the above result may be used to provide an alternative maximal dwell-time result in which the anti-stability constraint on the matrix $A$ is relaxed:
\begin{corollary}[Alternative maximal dwell-time result]
    Assume that for some given $T>0$, there exist matrices $P,Z\in\mathbb{S}^n_{++}\succ0$, $U,Q\in\mathbb{S}^n$, $N\in\mathbb{R}^{n\times 2n}$ and $R\in\mathbb{R}^{n\times n}$ such that the LMIs $\Psi(T)\prec0$, $\Phi(T)\prec0$ and $\mathscr{D}(P,J)\prec0$ hold.

  Then, for any impulse sequence $\{t_k\}_{k\in\mathbb{N}}$ satisfying $\eps<t_{k+1}-t_k\le T$, $\eps>0$, the system (\ref{eq:mainsyst}) is asymptotically stable and Lemma \ref{lem:maxDT} holds with the same matrix $P$.
\end{corollary}
\begin{proof}
  Using the ranged dwell-time result (Theorem \ref{th:aperiodic_range}) and letting $T\to0$, we obtain the limit LMI $J^TPJ-P\prec0$, as shown in Lemma \ref{lem:T0}. This concludes the proof.
\end{proof}
%


\subsection{Examples}

It is illustrated below, through academic examples, that the proposed alternative dwell-time formulations may lead to interesting results in terms of accuracy. It is important to stress that the conservatism is only due to the choice for the functional since Theorem \ref{Th0} provides equivalent statements. This increase of conservatism is the price to pay to obtain tractable and efficient tools for the analysis of uncertain linear impulsive systems.

\begin{example}[Maximal dwell-time]\label{ex:1}
  Let us consider system (\ref{eq:mainsyst}) with matrices
  \begin{equation}\label{eq:ex1}
  \begin{array}{lclclcl}
    A&=&\begin{bmatrix}
       1  & 3\\
      -1  & 2
    \end{bmatrix},&& J&=&0.5I_2.
  \end{array}
  \end{equation}
  Since $A$ has 2 unstable eigenvalues at $1.5\pm1.6583j$, then overall system stability requires an asymptotically stable $J$, which is the case here. Hence if the pulses do not occur frequently enough, stability is not achievable.

\textbf{Periodic impulses case:} In the periodic case, the spectral radius condition can be used and is equivalent to the feasibility of $\mathscr{I}(P,A,J,T)\prec0$ for some $P\in\mathbb{S}_{++}^n$. The spectral radius condition gives the maximal impulse period $T_{max}=2\log(2)/3\in(0.4620,0.4621)$. Theorem \ref{th:imp:nominal} is applied together with a bisection approach to find the maximal constant $T>0$ that preserves stability and yields the lower bound $T_{max}^\ell=0.4471$. This shows that the provided approach is able to estimate quite well the maximal pulse period for this example.

Choosing for instance $T=0.3$, we obtain the states and continuous-time Lyapunov function trajectories for system (\ref{eq:ex1}) depicted in Fig. \ref{fig:loc1} and \ref{fig:loc2}. 

\begin{figure}[h]
\centering
  \includegraphics[width=0.65\textwidth]{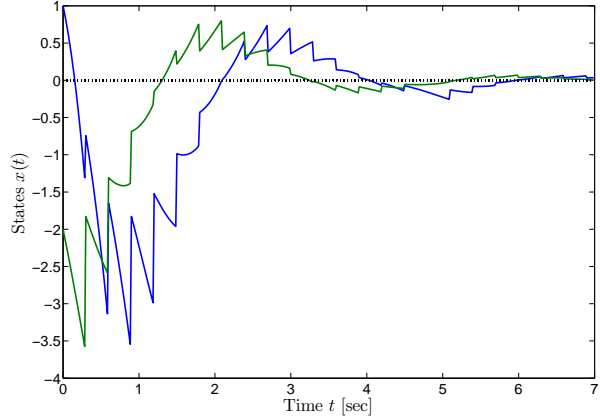}
  \caption{States trajectory of impulsive system (\ref{eq:ex1}) with impulse period $T=0.3$.}\label{fig:loc1}
\end{figure}
%

\textbf{Aperiodic impulses case:} Since the matrix $A$ is anti-Hurwitz, then this system may fulfill conditions of the maximal dwell-times results. When Lemma \ref{lem:maxDT}, based on the discrete-time Lyapunov condition, is considered, the lower bound on the maximal dwell-time $T_{max}^{\ell}=0.4620$ is found. This shows that Lemma \ref{lem:maxDT} is able to exactly characterize the maximal dwell-time since the computed value is identical to the maximal admissible impulse period (periodic case). The system (\ref{eq:ex1}) is hence asymptotically stable for all $T_k\in(0,T_{max}]$. Theorem \ref{th:maxDT} is now considered and yields the lower bound on the maximal dwell-time $T_{max}^{\ell_2}=0.4471$. Hence, Theorem \ref{th:maxDT} provides, for this example, a good approximation of Lemma \ref{lem:maxDT}.
\end{example}

\begin{example}[Minimal dwell-time]\label{ex:2}
  Now consider system (\ref{eq:mainsyst}) with matrices
    \begin{equation}\label{eq:ex2}
  \begin{array}{lclclcl}
    A&=&\begin{bmatrix}
       -1& 0\\
       1 &-2
    \end{bmatrix},&& J&=&\begin{bmatrix}
      2 &1\\
      1 &3
    \end{bmatrix}.
  \end{array}
  \end{equation}
    Since the matrix $A$ is Hurwitz and $J$ is anti-Schur, then, if the impulses are too frequent, stability is lost. It is hence expected to find a sufficiently small dwell-time for which stability is lost.

\textbf{Periodic impulses case:} The spectral radius condition yields $T_{min}\in[1.1405, 1.1406]$ while Theorem \ref{th:imp:nominal} gives the upper-bound $T_{min}^u=1.2323$ on the minimal impulse period.

\textbf{Aperiodic impulses case:} Since $A$ is Hurwitz, then the system may fulfill conditions of the minimal dwell-time results. Lemma \ref{lem:minDT} yields the upper bound on the minimal dwell-time $1.1405$ which is almost identical to the minimal period of the periodic case. This emphasizes the exactness of Lemma \ref{lem:minDT} in the estimation of the minimal dwell-time. System (\ref{eq:ex2}) is hence guaranteed to be asymptotically stable for all $T_k\in[T_{min}^u,+\infty)$. In contrast, Theorem \ref{th:minDT} predicts the upper bound  $T_{min}^u=1.2323$, showing the quite good accuracy of the proposed approach.
\end{example}

\begin{example}\label{ex:3}
  Consider now system (\ref{eq:mainsyst}) with matrices
  \begin{equation}\label{eq:ex3}
      \begin{array}{lclclcl}
    A&=&\begin{bmatrix}
       -1 & 0.1\\
       0 & 1.2
    \end{bmatrix},&& J&=&\begin{bmatrix}
      1.2 &0\\
      0 &0.5
    \end{bmatrix}.
  \end{array}
  \end{equation}
  Note that, the continuous-time dynamics of the first state is stable while the second is unstable. Conversely, the matrix $J$ has a stable eigenvalue for the second state and an unstable one for the first state. It is hence expected that the range of admissible dwell-times is a connected interval excluding $0$ and $+\infty$. It is important to stress that such a system cannot be analyzed using existing methods, such as the ones in \citep{Cai:08,Hespanha:08}, since neither $A$ nor $J$ is stable.

\textbf{Periodic impulses case:} An eigenvalue analysis gives the admissible range $[0.1824, 0.5776]$ of dwell-times. When Theorem \ref{th:imp:nominal} is used with two bisection routines, the following bounds are obtained $T_{min}^u = 0.1824$, $T_{max}^\ell = 0.5760$. This illustrates that the approach is able to characterize quite well the stability of this system.

To emphasize that this case cannot be handled with existing methods, let $T=0.3$ and we find $P=\diag(2.3622, 1.4752)$ along with \citep{Hespanha:08}:
\begin{equation}\label{eq:alphastab}
\begin{array}{lcl}
  \mathscr{C}(P,A)&\preceq&-cP\\
  \mathscr{D}(P,J)&\preceq& (e^{-d}-1)P
\end{array}
\end{equation}
where $c=-2.4036$ and $d=-0.3646$. Note also that no $P\succ0$ satisfies one of the conditions (\ref{eq:alphastab}) with $c>0$ or $d>0$. Since both $c$ and $d$ are negative for all $P\succ0$, the method of \citep{Hespanha:08} is inconclusive. This demonstrates the potential of the proposed functional-based approach.

\textbf{Aperiodic impulses case:} Theorem \ref{th:aperiodic_range} yields the interval $[0.1907, 0.5063]$, which is included in the interval obtained for the case of periodic impulses.\\

This example demonstrates that the proposed approach is able to capture the internal structure of this system better than other approaches, and take the advantage of this feature to improve accuracy. The main drawback of the existing methods based on a Lyapunov conditions and $\alpha$-stability of the form (\ref{eq:alphastab}) is that they consider the worst case eigenvalue (covering of the system) and lose/ignore the internal structure of the system, notably eigenvectors.
\end{example}

\begin{example}
  Let us consider now the sampled-data control system
  \begin{equation}
  \begin{array}{lcl}
    \dot{x}(t)&=&\tilde{A}x(t)+Bu(t),\\
    u(t)&=&Kx(t_k),\ t\in[t_k,t_{k+1})
  \end{array}
  \end{equation}
Reformulating this system as an impulsive system, we obtain
  \begin{equation}
  \begin{array}{lcl}
        \dot{z}(t)&=&\begin{bmatrix}
      \tilde{A} & B\\
      0 & 0
    \end{bmatrix}z(t)\\
    z^+(t_k)&=&\begin{bmatrix}
      I_n & 0\\
      K & 0
    \end{bmatrix}z(t_k)
  \end{array}
  \end{equation}
  where $z(t)=\col(x(t),u(t))$. Note that neither $A$ nor $J$ is a stable matrix, hence the developed maximal and minimal dwell-time results cannot be applied. The ranged dwell-time result however applies here.  It is well-known that when the corresponding continuous-time system is asymptotically stable, i.e. $\tilde{A}+BK$ Hurwitz, then the sampled-data system is stable in a sufficiently small positive neighborhood of the 'zero sampling-period' \cite{Oishi:10, Briat:11b,Seuret:12}. Unfortunately, this result cannot be used in the ranged dwell-time result (Theorem \ref{th:aperiodic_range}) since this condition cannot be obtained using the current approach that requires $J$ to be Schur.
%
%

Let us consider now a sampled-data system with matrices \citep{Naghshtabrizi:08,Seuret:12,Briat:11c}
  \begin{equation}
    \begin{array}{lclclclclcl}
      \tilde{A}&=&\begin{bmatrix}
        0 & 1\\
        0 & -0.1
      \end{bmatrix},\ &&B&=&\begin{bmatrix}
        0\\
        0.1
      \end{bmatrix},\ && K&=&-\begin{bmatrix}
        3.75 & 11.5
      \end{bmatrix}.
    \end{array}
  \end{equation}

Applying Theorem \ref{th:imp:nominal} we find the lower bound on the maximal constant sampling period $T_{max}^\ell=1.7239$ similarly as in \citep{Seuret:12,Briat:11c}. Comparatively, the spectral radius condition yields a maximal value of $1.7294$. Focusing now on the aperiodic case, we apply Theorem \ref{th:aperiodic_range} where we set $T_{min}=10^{-5}$ and we find the lower bound $T_{max}^\ell= 1.7239$.
\end{example}

\section{Quadratic stability analysis of aperiodic uncertain linear impulsive systems}\label{sec:rob}

Extensions of the previous results to the case of uncertain systems are discussed in this section. We hence now consider the case where the matrices $A$ and $J$ belong to some distinct polytopes $\mathcal{A}$ and $\mathcal{J}$, as defined in (\ref{eq:polytopes}). Due to the affine structure of the conditions, the extension to uncertain systems is immediate. Despite of this ease, most of the efficiency of the approach is preserved as demonstrated in what follows.

\subsection{Main results}

The cases of periodic impulses and minimal, maximal and ranged dwell-times for uncertain systems are discussed.

\begin{theorem}[Periodic impulses]\label{th:newrobth}
  The uncertain linear impulsive system (\ref{eq:mainsyst})-(\ref{eq:polytopes}) with periodic impulses is asymptotically stable if there exist matrices $P,Z_j\in\mathbb{S}_{++}^n$, $Q_j,U_j\in\mathbb{S}^n$, $R_j\in\mathbb{R}^{n\times n}$ and $N_j\in\mathbb{R}^{n\times 2n}$ for $j=1,\ldots,N_J$ such that the LMIs
\begin{equation}\label{eq:imp_robA}
\begin{array}{rcl}
  \Psi_{ij}(T):=\begin{bmatrix}
    G^0_{ij}+T(G^1_{ij}+G^2_j) & M_x^TA_i^TZ_j\\
    \star & -Z_j/T
  \end{bmatrix}&\prec&0,\\
  \Phi_{ij}(T):=\begin{bmatrix}
    G^0_{ij}-TG_j^2 & N_j^T\\
    \star & -Z_j/T
  \end{bmatrix}&\prec&0
  \end{array}
\end{equation}
hold for all $i=1,\ldots,N_A$ and $j=1,\ldots,N_J$ where
\begin{equation*}
  \begin{array}{lcl}
    G^0_{ij}&=&T\cdot\He[M_x^TA_i^TPM_x]-M_j^TQ_jM_j\\
    &&+M_-^T(J_j^TPJ_j-P)M_-+\He[N_j^TM_j-M_j^TR_jM_x],\\
    G^1_{ij}&=&\He[M_x^TA_i^TQ_jM_j+M_x^TA_i^TR_jM_x+M_j^TR_jA_iM_x],\\
    G^2_j&=&M_-^TJ_j^TU_jJ_jM_-.
\end{array}
\end{equation*}
Moreover, the quadratic form $V(x)=x^TPx$ is a quadratic Lyapunov function for the uncertain system (\ref{eq:mainsyst}), that is, we have $\mathscr{I}(P,A,J,T)\prec0$ for all $(A,J)\in\mathcal{A}\times\mathcal{J}$.
\end{theorem}
\begin{proof}
The goal is to show that the feasibility of the LMIs $\Psi_{ij}(T)\prec0$ and $\Phi_{ij}(T)\prec0$, for all $i=1,\ldots,N_A$ and all $j=1,\ldots,N_J$ implies the feasibility of
$\mathscr{I}(P,A,J,T)\prec0$ for all $(A,J)\in\mathcal{A}\times\mathcal{J}$.

To show this, multiply first $\Psi_{ij}$ and $\Phi_{ij}$ by $\kappa^1_i$ where $\kappa^1=\col_{i=1}^{N_A}[\kappa^1_i]$ belongs to the unit $N_A$-simplex. The sums $\sum_{i=1}^{N_A}\kappa^1_i\Psi_{ij}$ and $\sum_{i=1}^{N_A}\kappa^1_i\Phi_{ij}$ are then both negative definite since the $\Psi_{ij}$'s and $\Phi_{ij}$'s are individually negative definite. After a Schur complement, we obtain an LMI of the same form as in Theorem \ref{th:imp:nominal}, implying then that for each $j=1,\ldots,N_J$  we have
\begin{equation*}
  \mathscr{I}(P,A,J_j,T)\prec0
\end{equation*}
for all $A\in\mathcal{A}$. Note that, by virtue of the Schur complement formula, the above LMI is equivalent to
\begin{equation}
\begin{bmatrix}
  -P & J_j^Te^{A^TT}P\\
  \star & -P
\end{bmatrix}\prec0
\end{equation}
which is linear in $J_j$. Multiplying the above LMI by $\kappa_j^2$ where $\kappa^2=\col_{j=1}^{N_J}[\kappa^2_i]$ belongs to the unit $N_J$-simplex and summing over $j=1,\ldots,N_J$ yields the condition
\begin{equation}
  \mathscr{I}(P,A,J,T)\prec0,\ \mathrm{for\ all\ }(A,J)\in\mathcal{A}\times\mathcal{J}
\end{equation}
by using again the Schur complement formula. This concludes the proof.
\end{proof}

Following the same reasoning as in the previous section, the above theorem can be used to easily derive dwell-time results for uncertain systems.

%
%
%

%

\begin{theorem}[Robust minimal dwell-time]\label{th:minDT_A}
  Assume that for some given $T>0$, there exist matrices $P,Z_j\in\mathbb{S}^n_{++}\succ0$, $U_j,Q_j\in\mathbb{S}^n$, $N_j\in\mathbb{R}^{n\times 2n}$ and $R_j\in\mathbb{R}^{n\times n}$ such that $\Psi_{ij}(T)\prec0$, $\Phi_{ij}(T)\prec0$ and $\mathscr{C}(P,A_i)\prec0$ hold for all $i=1,\ldots,N_A$ and all $j=1,\ldots,N_J$.

Then, for any impulse sequence $\{t_k\}_{k\in\mathbb{N}}$ satisfying $t_{k+1}-t_k\ge T$, the uncertain linear impulsive system (\ref{eq:mainsyst})-(\ref{eq:polytopes}) is asymptotically stable. That is we have $\mathscr{I}(P,A,J,\theta)\prec0$ for all $(A,J)\in\mathcal{A}\times\mathcal{J}$ and all $\theta\in[T,+\infty)$.
\end{theorem}

\begin{theorem}[Robust maximal dwell-time]\label{th:maxDT_A}
  Assume that for some given $T>0$, there exist matrices $P,Z_j\in\mathbb{S}^n_{++}\succ0$, $U_j,Q_j\in\mathbb{S}^n$, $N_j\in\mathbb{R}^{n\times 2n}$ and $R_j\in\mathbb{R}^{n\times n}$ such that $\Psi_{ij}(T)\prec0$, $\Phi_{ij}(T)\prec0$ and $\mathscr{C}(P,A_i)\succ0$ hold for all $i=1,\ldots,N_A$ and all $j=1,\ldots,N_J$.

Then, for any impulse sequence $\{t_k\}_{k\in\mathbb{N}}$ satisfying $\eps<t_{k+1}-t_k\le T$, $\eps>0$, the uncertain linear impulsive system (\ref{eq:mainsyst})-(\ref{eq:polytopes}) is asymptotically stable. That is we have $\mathscr{I}(P,A,J,\theta)\prec0$ for all $(A,J)\in\mathcal{A}\times\mathcal{J}$ and all $\theta\in(0,T]$.
\end{theorem}

\begin{theorem}[Robust ranged dwell-time]\label{th:imp_robA}
Assume that for some $\eps<T_{min}<T_{max}<+\infty$, $\eps>0$, there exist matrices $P,Z_j\in\mathbb{S}_{++}^n$, $Q_j,U_j\in\mathbb{S}^n$, $R_j\in\mathbb{R}^{n\times n}$ and $N_j\in\mathbb{R}^{n\times 2n}$ such that $\Psi_{ij}(T)\prec0$ and $\Phi_{ij}(T)\prec0$ hold for all $T\in\{T_{min},T_{\max}\}$, all $i=1,\ldots,N_A$ and all $j=1,\ldots,N_J$.

Then, for any impulse sequence $\{t_k\}_{k\in\mathbb{N}}$ satisfying $t_{k+1}-t_k\in[T_{min},T_{max}]$, the uncertain linear impulsive system (\ref{eq:mainsyst})-(\ref{eq:polytopes}) is asymptotically stable. That is we have $\mathscr{I}(P,A,J,\theta)\prec0$ for all $(A,J)\in\mathcal{A}\times\mathcal{J}$ and all $\theta\in[0,T]$.
\end{theorem}

\subsection{Examples}

\begin{example}[Robust Maximal Dwell-Time]
  Let us consider an uncertain version of the system treated in Example \ref{ex:1}. Assume that the matrix $J$ is known and that the matrix $A$ belongs to
\begin{equation*}
  \mathcal{A}:=\co\left\{\begin{bmatrix}
    1 &3\\-1& 2
  \end{bmatrix},\begin{bmatrix}
    2 &2\\0& 6
  \end{bmatrix}\right\}.
\end{equation*}

\textbf{Periodic impulses case:} An eigenvalue analysis performed over the set of all possible systems locates the maximal impulse period inside the interval $[0.1155,0.1156]$. By gridding the LMI with a very thin grid $\mathscr{I}(P,A,J,T)\prec0$, we obtain the lower bound $0.1155$ on the maximal admissible period, at the expense of a very high computational cost, gridding imprecision and large computation time. In contrast, Theorem \ref{th:newrobth} yields the lower bound on the maximal period $T_{max}^\ell=0.1148$ at a much lower computational cost and infinite precision (since the entire interval of dwell-times is considered rather than a finite subset of it), but with a slight conservatism.

\textbf{Aperiodic impulses case:} By gridding the conditions of Lemma \ref{lem:maxDT}, the lower bound $0.1155$ on the maximal dwell-time is obtained. When Theorem \ref{th:imp_robA} is used, we obtain $T_{max}^\ell=0.1148$. This demonstrates the interest of functionals for analyzing the stability of uncertain linear impulsive systems.
\end{example}

\begin{example}[Robust Ranged Dwell-Time]
  Example \ref{ex:3} is revisited here with the difference that $A$ is known while $J\in\mathcal{J}$ where
  \begin{equation*}
  \mathcal{J}:=\co\left\{\begin{bmatrix}
    1.3 & 0\\
    0 & 0.25
  \end{bmatrix},\begin{bmatrix}
    1.1 & 0\\
    0 & 0.5
  \end{bmatrix}\right\}.
\end{equation*}
An eigenvalue analysis done over the entire set of possible systems yields the following interval of admissible impulse periods $[0.2624,0.5776]$. In contrast, Theorem \ref{th:newrobth} yields the upper and lower bounds on the admissible impulse periods $T_{min}^u = 0.2625$ and $T_{max}^\ell=0.5761$. These bounds immediately generalize to the aperiodic case.
\end{example}

\section{Conclusion}

A looped-functional-based approach has been proposed for the analysis of uncertain linear impulsive systems. The main advantage of the proposed approach lies in the expression of a discrete-time stability criterion using a continuous-time approach. Such a framework allows one to easily consider non-monotonic continuous-time Lyapunov functions and is able to analyze a wider class of systems than existing approaches. As a byproduct, this approach leads to an exponential-free formulation of the stability conditions, allowing to obtain robustness result very easily. Several examples illustrate the potential of the approach.

The proposed methodology can be applied to many other types of continuous-time and hybrid systems for which a recursive discrete-time stability criterion can be applied. Furthermore, since the current formulation involves continuous-time data, the framework can be readily extended to nonlinear systems, higher order Lyapunov functions, the use of $T$-dependent Lyapunov functions, etc. These problems are way beyond the scope of this paper and are left for future works.

\bibliographystyle{elsarticle-harv}

\end{document}